\documentclass{amsart}
\usepackage[latin1]{inputenc}
\usepackage[english]{babel}
\usepackage{csquotes}
\usepackage{todonotes}
\usepackage{amsthm}
\usepackage{verbatim}
\usepackage{amsmath,amssymb}
\usepackage{graphicx}
\usepackage{amsthm}
\usepackage{scalerel,stackengine}
\stackMath
\newcommand\lift[1]{%
\savestack{\tmpbox}{\stretchto{%
  \scaleto{%
    \scalerel*[\widthof{\ensuremath{#1}}]{\kern-.6pt\bigwedge\kern-.6pt}%
    {\rule[-\textheight/2]{1ex}{\textheight}}
  }{\textheight}%
}{0.5ex}}%
\stackon[1pt]{#1}{\tmpbox}%
}
\usepackage{color} 
\usepackage[colorlinks=true,linktoc=all,linkcolor=blue,citecolor=red,filecolor=pink,urlcolor=blue]{hyperref}
\usepackage[hyperref=false,style=alphabetic,citestyle=alphabetic,backend=bibtex,maxnames=100,doi=false,isbn=false,url=false,giveninits=true]{biblatex}
\bibliography{bibli}

\usepackage{amssymb,amsmath,latexsym}

\theoremstyle{plain}
\newtheorem{theorem}{Theorem}[section]
\newtheorem{prop}[theorem]{Proposition}
\newtheorem{lemma}[theorem]{Lemma}
\newtheorem{defin}[theorem]{Definition}
\newtheorem{cor}[theorem]{Corollary}
\newtheorem{oss}[theorem]{Remark}
\newtheorem{cl}[theorem]{Claim}
\numberwithin{equation}{section}

\begin{document}
\title{On the topology of surfaces with the generalised simple lift property}

\author{Francesca Tripaldi}
\address{Department of Mathematics and Statistics, University of Jyv\"{a}skyl\"{a}, Jyv\"{a}skyl\"{a} FI-40014, Finland}
\email{francesca.f.tripaldi@jyu.fi}

\subjclass[2010]{53A10, 51H05}

\begin{abstract}
In this paper, we study the geometry of surfaces with the generalised simple lift property. This work generalises previous results by Bernstein and Tinaglia \cite{BeT} and it is motivated by the fact that leaves of a minimal lamination obtained as a limit of a sequence of property embedded minimal disks satisfy the generalised simple lift property.
\end{abstract}

 \maketitle
 \tableofcontents

\section{Introduction}
Motivated by the work of Colding and Minicozzi \cite{cm21,cm22,cm23,cm24} and Hoffman and White \cite{hofw1} on minimal laminations obtained as limits of sequences of properly embedded minimal disks, in~\cite{BeT} Bernstein and Tinaglia  introduce the concept of the simple lift property. Interest in these surfaces arises because leaves of a minimal lamination obtained as a limit of a sequence of properly embedded minimal disks satisfy the simple lift property. In~\cite{BeT} they prove that an embedded minimal surface $\Sigma\subset\Omega$ with the simple lift property must have genus zero, if $\Omega$ is an orientable three-manifold satisfying certain geometric conditions. In particular, one key condition is that $\Omega$ cannot contain closed minimal surfaces.

In this paper, we generalise this result by taking an arbitrary orientable three-manifold $\Omega$ and introducing the concept of the generalised simple lift property, which extends the simple lift property in \cite{BeT}. Indeed, we prove that leaves of a minimal lamination obtained as a limit of a sequence of properly embedded minimal disks satisfy the generalised simple lift property and we are able to restrict the topology of an arbitrary surface $\Sigma\subset\Omega$ with the generalised simple lift property. 

Among other things, we prove that the only possible compact leaves of a minimal lamination obtained as limits of a sequence of properly embedded minimal disks are the torus in the orientable case, the Klein bottle and the connected sum of three and four projective planes in the non-orientable case. 

\vspace{.5cm}
\textbf{Acknowledgements}: the author wishes to thank her advisor Giuseppe Tinaglia for his help and support during her PhD at King's College London. This work has been partially supported by the Academy of Finland (grant 288501 `\emph{Geometry of subRiemannian groups}') and by the European Research Council  (ERC Starting Grant 713998 GeoMeG `\emph{Geometry of Metric Groups}').


\section{Notation and definitions}

Throughout the paper, we will assume $\Omega$ to be an open subset of an orientable three-dimensional Riemannian manifold $(M,g)$. We denote by ${dist}^\Omega$ the distance function on $\Omega$ and by $\exp^\Omega$ the exponential map. Therefore, we have
\begin{align*}
\exp^\Omega_p:\mathbb{B}_r(0)\to\mathcal{B}_r(p)\,,
\end{align*}
where $\mathbb{B}_r(0)$ is the Euclidean ball in $\mathbb{R}^3$ of radius $r$ centred at the origin, and $\mathcal{B}_r(p)$ is the geodesic ball in $M$ of radius $r$ centred at $p\in \Omega$.

For an embedded surface $\Sigma$, we write 
\begin{align*}
\exp^\perp:N\Sigma\to\Omega
\end{align*}
to denote the normal exponential map, where $N\Sigma$ is the normal bundle.

If $N\Sigma$ is trivial, then we say that $\Sigma$ is \textit{two-sided}, otherwise we say that $\Sigma$ is \textit{one-sided}. As $\Omega$ is oriented, $\Sigma$ being two-sided is equivalent to saying that $\Sigma$ is orientable.

Let us fix a subset $U\subset N\Sigma$, then we define
\begin{align*}
\mathcal{N}_U(\Sigma):=\exp^\perp(U)\,.
\end{align*}

The set $\mathcal{N}_U(\Sigma)$ is \textit{regular} if there is an open set $V$ with $U\subset V$ such that $\exp^\perp:V\to\mathcal{N}_V(\Sigma)$ is a diffeomorphism. If $\mathcal{N}_U(\Sigma)$ is regular, then the map $\Pi_\Sigma:\mathcal{N}_U(\Sigma)\to\Sigma$, given by the nearest point projection, is smooth and for any $(q,\mathbf{v})\in T\mathcal{N}_U(\Sigma)$, there is a natural splitting 
\begin{align*}
\mathbf{v}=\mathbf{v}^\perp+\mathbf{v}^T\,,
\end{align*}
where $\mathbf{v}^\perp$ is orthogonal to $\mathbf{v}^T$, and $\mathbf{v}^T$ is perpendicular to the fibres of $\Pi_\Sigma$.

We say that such $\mathbf{v}$ is \textit{$\delta$-parallel} to $\Sigma$ if
\begin{align*}
\vert\mathbf{v}^\perp\vert\le\delta\vert\mathbf{v}\vert\;\text{ and }\;\frac{1}{1+\delta}\vert\mathbf{v}^T\vert\le\vert d(\Pi_\Sigma)_q(\mathbf{v})\vert\le (1+\delta)\vert\mathbf{v}^T\vert\,.
\end{align*}

Given $\epsilon>0$, we set $U_\epsilon:=\lbrace (p,\mathbf{v})\in N\Sigma\mid \vert\mathbf{v}\vert<\epsilon\rbrace$ and define $\mathcal{N}_\epsilon(\Sigma)$, the $\epsilon$-neighbourhood of $\Sigma$, to be $\mathcal{N}_{U_\epsilon}(\Sigma)$. If $\Sigma$ is an embedded smooth surface and $\Sigma_0\subset\Sigma$ is a pre-compact subset, then $\exists\,\epsilon>0$ so that $\mathcal{N}_\epsilon(\Sigma_0)$ is regular.

Given a fixed embedded surface $\Sigma$ and $\delta\ge 0$, we say that another embedded smooth surface $\Gamma$ is a \textit{smooth $\delta$-graph} over $\Sigma$ if there exists an $\epsilon>0$ such that:
\begin{itemize}
\item[1.] $\mathcal{N}_\epsilon(\Sigma)$ is a regular $\epsilon$-neighbourhood of $\Sigma$;
\item[2.] either $\Gamma$ is a proper subset of $\mathcal{N}_\epsilon(\Sigma)$ or $\Gamma$ is a proper subset of $\mathcal{N}_\epsilon(\Sigma)\setminus\Sigma$;
\item[3.] for all $(q,\mathbf{v})\in T\Gamma $ is $\delta$-parallel to $\Sigma$.
\end{itemize}

We say that a smooth $\delta$-graph $\Gamma$ over $\Sigma$ is a \textit{smooth $\delta$-cover} of $\Sigma$, if $\Gamma$ is connected and $\Pi_\Sigma(\Gamma)=\Sigma$.

Let $\gamma:[0,1]\to\Sigma$ be a smooth curve in $\Sigma$. We will also denote the image of such $\gamma$ as $\gamma$.

We say that a curve $\lift{\gamma}:[0,1]\to\mathcal{N}_\delta(\gamma)$ is a $\delta$\textit{-lift} of $\gamma$ if
\begin{itemize}
\item $\mathcal{N}_\delta(\gamma)$ is regular;
\item $\Pi_\Sigma\circ\lift{\gamma}=\gamma$;
\item for all $t\in [0,1]$, $(\lift{\gamma}(t),\lift{\gamma}'(t))$ is $\delta$-parallel to $\Sigma$.
\end{itemize}

This definition extends to piece-wise $C^1$ curves in an obvious manner.

\section{The generalised simple lift property for a finite number of curves}

Let us introduce the concept of lifts of curves into embedded disks.
\begin{defin}\label{genliftprop} Generalised simple lift property.

Let $\Sigma$ be a surface in $\Omega$. Then $\Sigma$ has the \textit{generalised simple lift property} if, for any $\delta >0$ and for any $p\in\Sigma$, the following holds.

Given $\gamma_1,\ldots,\gamma_n:[0,1]\to\Sigma$ a collection of $n$ arbitrary smooth curves, and any pre-compact open subset $U\subset\Sigma$ such that $\gamma_1\cup\cdots\cup\gamma_n\subset U$, there exist $t_i\in [0,1]$ for which $\gamma_i(t_i)=p$ for any $i=1,\ldots,n$, as well as:
\begin{itemize}
\item[i.] a constant $\epsilon=\epsilon(U,\delta)>0$;
\item[ii.] $\Delta\subset\Omega$ an embedded disk;
\item[iii.] $\lift{\gamma}_i:[0,1]\to \mathcal{N}_\epsilon(U)$ $\delta$-lifts of $\gamma_i$
\end{itemize}
such that 
\begin{itemize}
\item[1.] $\lift{\gamma}_i\subset\Delta\cap\mathcal{N}_\epsilon(U)$;
\item[2.] $\Delta\cap\mathcal{N}_\epsilon(U)$ is a $\delta$-graph over $U$;
\item[3.] there exists a point $q\in\mathcal{N}_\epsilon(p)\cap\Delta$ such that $q\in\lift{\gamma}_i$ for every $i=1,\ldots,n$;
\item[4.] the connected component of $\Delta\cap\mathcal{N}_\epsilon(U)$ containing $\lift{\gamma}_i$ is a $\delta$-cover of $U$.
\end{itemize}

 The union $\lift{\gamma}_1\cup\cdots\cup\lift{\gamma}_n$ is called the \textit{generalised simple} $\delta$\textit{-lift of} $\gamma_1\cup\cdots\cup\gamma_n$ \textit{pointed at}  $(p,q)$ \textit{into} $\Omega$.
 
One should notice that the embedded disk $\Delta\subset\Omega$ that the definition implies exists will depend on the choice of the constant $\delta>0$, the $n$ curves $\gamma_1,\ldots,\gamma_n$ and the pre-compact subset $U\subset\Sigma$ that contains the curves. Notation wise, throughout this paper, when studying a lift of $n$ given curves $\gamma_1,\ldots,\gamma_n$, if we want to highlight the dependence of the construction on the choice of curves, we will denote the embedded disk $\Delta$ that contains the generalised simple $\delta$-lift of $\gamma_1\cup\cdots\cup\gamma_n$ by $\Delta(\gamma_1,\ldots,\gamma_n)$.
\end{defin}

A surface with the generalised simple lift property is one for which, in an effective sense, the universal cover of the surface can be properly embedded as a disk near the surface. For this reason, to understand the topology of the surface $\Sigma$, it is important to understand the lifting behaviour of closed curves. 

With this in mind, we give the following definition.

\begin{defin} Closed and open lift property.\\
\indent Let $\Sigma\subset \Omega$ be an embedded surface with the generalised simple lift 
property. If $\gamma : [0,1]\to \Sigma$ is a smooth closed curve, 
then $\gamma$ has the \emph{open lift property} if there  exists a $\delta_0>0$ 
so that, for all  $\delta_0>\delta>0$, $\gamma$ does not have a closed generalised simple 
$\delta$-lift $\lift{\gamma}:[0,1]\to \mathcal{N}_\delta(\Sigma)$.  Otherwise, 
$\gamma$ has the \emph{closed lift property}.

If a closed curve $\gamma$ has the closed lift property, then there is a sequence $\delta_i\to 0$ so that there are closed simple $\delta_i$-lifts $\lift{\gamma}_i$ of $\gamma$.  

If it is possible to choose the lifts of a curve $\gamma$ to be embedded (and in particular non-intersecting), we say $\gamma$ has the \emph{embedded (closed/open) lift property}.
\end{defin}

\begin{oss}
In Proposition \ref{Lifting Lemma} below and in Lemma \ref{no_two_closed} we will be constructing the lift of two (or more) simple closed curves intersecting at one point by considering the union of these curves as a single curve.

In order to fix the notation, let us assume we want to lift two curves $\alpha,\,\beta:[0,1]\to\Sigma$ intersecting in one point $\lbrace p\rbrace=\alpha\cap\beta$. Then we will denote by $\mu:=\alpha\circ\beta$ the curve parametrised as $\mu:[0,1]\to\Sigma$ with
\begin{align*}
\mu\big\vert_{[0,t_0]}=\tilde{\alpha}\,,\;\mu\big\vert_{[t_0,1]}=\tilde{\beta}
\end{align*}
where $\tilde{\alpha}=\alpha(\frac{t}{t_0})\colon [0,t_0]\to\Sigma$ and $\tilde{\beta}=\beta(\frac{t-t_0}{1-t_0})\colon[t_0,1]\to\Sigma$ are appropriate reparametrisations of $\alpha$ and $\beta$ respectively, so that $\mu(0)=\mu(t_0)=\mu(1)$, for some $t_0\in (0,1)$.
\end{oss}

The proposition below, which we will call \textit{Lifting Lemma}, is analogous to Proposition 4.4 in Bernstein and Tinaglia's paper \cite{BeT}. 
\begin{prop}{\textbf{Lifting lemma}}\label{Lifting Lemma}\\
\indent Let $\Sigma\subset\Omega$ be an embedded surface with the generalised simple lift property. Let us take into consideration two closed, smooth curves
\begin{align*}
\alpha:[0,1]\to \Sigma\;\text{ and }\;\beta:[0,1]\to \Sigma
\end{align*}
satisfying the following properties:
\begin{itemize}
\item[1.] $\alpha\cap\beta=\lbrace p\rbrace$, where $p=\alpha(0)=\beta(0)$;
\item[2.] $\exists\,U\subset \Sigma$ a two-sided pre-compact open set that contains both curves, i.e. $\alpha\cup\beta\subset U$;
\item[3.] for this choice of $U\subset\Sigma$, there exists a $\delta>0$ for which the embedded disk $\Delta=\Delta(\alpha,\beta)\subset\Omega$ given by Definition \ref{genliftprop} contains open lifts for both $\alpha$ and $\beta$.
\end{itemize}
Then the curve $\mu:=\alpha\circ \beta\circ\alpha^{-1}\circ\beta^{-1}$ has a closed lift into this disk $\Delta(\alpha,\beta)$.

If, in addition, both $\alpha$ and $\beta$ have embedded open lifts in $\Delta(\alpha,\beta)$, then one of the following curves has an embedded closed lift in $\Delta(\alpha,\beta)$:
\begin{align*}
\mu\;,\;\alpha\circ \beta\;,\;\beta\circ\alpha^{-1}\,.
\end{align*}
\end{prop}
\begin{proof}
Let us take into consideration the curve $\mu=\alpha\circ\beta\circ\alpha^{-1}\circ\beta^{-1}$ as defined above.

Since $\Sigma$ has the generalised simple lift property, then for any $\delta >0$ there exist:
\begin{itemize}
\item[i.] a positive constant $\epsilon>0$;
\item[ii.] an embedded disk $\Delta=\Delta(\alpha,\beta)$;
\item[iii.] $\lift{\mu}:[0,1]\to \mathcal{N}_\delta(U)$ a $\delta$-lift of $\mu$;
\end{itemize}
such that $\Delta(\alpha,\beta)\cap\mathcal{N}_\epsilon(U)$ is a $\delta$-graph over $U$, $\lift{\mu}\subset\Delta$, and $\Gamma$, the connected component of $\Delta(\alpha,\beta)\cap\mathcal{N}_\epsilon(U)$ containing $\lift{\mu}$, is a $\delta$-cover of $U$.

By assumption, we can consider the embedded disk $\Delta(\alpha,\beta)$ which contains opens lifts of both $\alpha$ and $\beta$.

By re-parametrising appropriate restrictions of $\lift{\mu}$, we can write $\lift{\mu}=\lift{\alpha}\circ \lift{\beta}\circ\lift{\alpha^{-1}}\circ \lift{\beta^{-1}}$, where the $\lift{\alpha},\,\lift{\beta},\,\lift{\alpha^{-1}},\,\lift{\beta^{-1}}:[0,1]\to \Gamma$ are the $\delta$-lifts of $\alpha,\,\beta,\,\alpha^{-1}$ and $\beta^{-1}$ respectively.

Let us now pick a small simply-connected neighbourhood $V$ of the point $p=\mu(0)$ such that $V\subset U$. By construction, $\Delta(\alpha,\beta)$ is an embedded disk, which means that we can order by height the components of $\Pi_\Sigma^{-1}(V)\cap\Delta(\alpha,\beta)$, where $\Pi_\Sigma$ is the usual projection map onto $\Sigma$. We will denote these ordered components as $\Pi_\Sigma^{-1}(V)\cap\Delta(\alpha,\beta)=\lbrace \lift{V}(1),\ldots,\lift{V}(n)\rbrace\,$. The number $n$ of components will of course depend on the choice of $\delta>0$ and $\Delta(\alpha,\beta)$.

By construction, we then have:
\begin{align*}
\Pi_\Sigma\big(\,\lift{\alpha(0)}\,\big)&=\Pi_\Sigma\big(\,\lift{\alpha(1)}\,\big)=\Pi_\Sigma\big(\,\lift{\beta(0)}\,\big)=\Pi_\Sigma\big(\,\lift{\beta(1)}\,\big)=\Pi_\Sigma\big(\,\lift{\alpha^{-1}(0)}\,\big)=\\&=\Pi_\Sigma\big(\,\lift{\alpha^{-1}(1)}\,\big)=\Pi_\Sigma\big(\,\lift{\beta^{-1}(0)}\,\big)=\Pi_\Sigma\big(\,\lift{\beta^{-1}(1)}\,\big)=p\,.
\end{align*}

Without loss of generality, one can assume $\lift{\mu}$ to be the generalised simple $\delta$-lift of $\mu$ pointed at $(p,q)$ with $q=\lift{\alpha(0)}$. Moreover, a priori, these points will all belong to different components of $\Pi_\Sigma^{-1}(V)\cap\Delta(\alpha,\beta)$ and we will denote them as:
\begin{itemize}
\item[] $\lift{p}(0):=\lift{\alpha(0)}=q\,$;
\item[] $\lift{p}(1):=\lift{\alpha(1)}=\lift{\beta(0)}\,$;
\item[] $\lift{p}(2):=\lift{\beta(1)}=\lift{\alpha^{-1}(0)}\,$;
\item[] $\lift{p}(3):=\lift{\alpha^{-1}(1)}=\lift{\beta^{-1}(0)}\,$;
\item[] $\lift{p}(4)=\lift{\beta^{-1}(1)}\,$;
\end{itemize}
so that $\lift{p}(j)\in\lift{V}(l)$, where $l$ is a function of $j$ over the natural numbers, that is $l=l(j)\in\mathbb{N}$.

Using this function $l$, we will study the signed number of sheets between the end points of the lifts of the curves $\alpha,\,\alpha^{-1},\,\beta$ and $\beta^{-1}$:
\begin{align*}
m[\alpha]:=l(1)-l(0);\\
m[\beta]:=l(2)-l(1);\\
m[\alpha^{-1}]:=l(3)-l(2);\\
m[\beta^{-1}]:=l(4)-l(3)\,.
\end{align*}

By assumption, both $\lift{\alpha}$ and $\lift{\beta}$ are open lifts, so that $m[\alpha],\,m[\beta]\neq 0$, which also implies $m[\alpha^{-1}],\,m[\beta^{-1}]\neq 0$.

We will now prove that $m[\alpha]=-m[\alpha^{-1}]$ and $m[\beta]=-m[\beta^{-1}]$, and therefore that $\lift{\mu}$ is closed.

Let us consider the two following cases separately:
\begin{itemize}
\item $m[\alpha]\cdot m[\beta] >0$
\end{itemize}

Without loss of generality, we can assume in this case that both numbers are positive: $m[\alpha],m[\beta] >0$. Then, using the fact that the disk $\Delta(\alpha,\beta)$ is embedded and that $U$ is two-sided, one can consider a disjoint family of \textit{parallel lifts} of $\alpha$, which we will denote by $\lift{\alpha}[i]$. The first member of this family is $\lift{\alpha}[0]=\lift{\alpha}$ and the subsequent representatives of the family are those lifts $\lift{\alpha}[i]$ of $\alpha$ such that $\lift{\alpha}[i](0)$ will belong to $\lift{V}\big(l(0)+i\big)$, which is the lift that starts $i$ sheets above $\lift{\alpha(0)}=q$. By the embeddedness of $\Delta(\alpha,\beta)$ and the two-sidedness of $U$, the signed number of graphs between $\lift{\alpha}[0](t)$ and $\lift{\alpha}[i](t)$ is constant in $t$, so that also the lifts $\lift{\alpha}[i]$ also have endpoints $i$ sheets above the endpoint of $\lift{\alpha}$.

Clearly, the lifts $\lift{\alpha}[i]$ are well-defined as long as $i\le m[\beta]$. Furthermore, $\lift{\alpha}[m[\beta]]$ has end point which is the same as the end point of $\lift{\beta}$. Let us now take into consideration $\lift{\alpha}[m[\beta]]^{-1}$. This is a lift of $\alpha^{-1}$ that starts at $\lift{\beta(1)}$, which means that $\lift{\alpha}[m[\beta]]^{-1}$ and $\lift{\alpha^{-1}}$ must coincide. This then implies that $m[\alpha]=-m[\alpha^{-1}]$. 

Repeating the same argument for $\lift{\beta}[-m[\alpha]]$ and $\lift{\beta}$ shows that $m[\beta]=-m[\beta^{-1}]$.
\begin{itemize}
\item $m[\alpha]\cdot m[\beta] <0$
\end{itemize}

In this case, we can assume without loss of generality that $m[\alpha]>0$ and $m[\beta]<0$.

Let us first assume that $m[\alpha]+m[\beta]+m[\alpha^{-1}]=:M\ge 0$, which means that the end point of $\lift{\alpha^{-1}}$ is not below the initial point of $\lift{\alpha}$: it is $M$ sheets above $\lift{\alpha}$. Repeating the same argument as in the previous case, we can take into consideration the parallel lift of $\lift{\alpha^{-1}}$ whose endpoint is the initial point of $\lift{\alpha}$, namely $\lift{\alpha^{-1}}[-M]$. The lift $\lift{\alpha^{-1}}[-M]^{-1}$ will then be a lift of $\alpha$ and it coincides with $\lift{\alpha}$, which implies as before that $m[\alpha]=-m[\alpha^{-1}]$. Therefore the initial assumption $m[\alpha]+m[\beta]+m[\alpha^{-1}]\ge 0$ leads to a contradiction, since by hypothesis $m[\beta]<0$.

It is then the case that $m[\alpha]+m[\beta]+m[\alpha^{-1}]=:M < 0$. Again, we can take into consideration the parallel lift of $\lift{\alpha}$ whose start point coincides with the endpoint of $\lift{\alpha^{-1}}$, namely $\lift{\alpha}[M]$. Therefore $\lift{\alpha}[M]^{-1}$ is the lift of $\alpha^{-1}$ with the same endpoint as $\lift{\alpha^{-1}}$: $\lift{\alpha}[M]^{-1}$ and $\lift{\alpha^{-1}}$ coincide, and so $m[\alpha]=-m[\alpha^{-1}]$. The same argument shows that in this case $m[\beta]=-m[\beta^{-1}]$.

Finally, if $\alpha$ and $\beta$ have embedded open lifts in $\Delta(\alpha,\beta)$, then, because they meet at only one point in $\Sigma$, the curves $\lift{\alpha}\circ \lift{\beta},\,\lift{\beta}\circ\lift{\alpha^{-1}}$ and $\lift{\alpha^{-1}}\circ \lift{\beta^{-1}}$ are all embedded. Hence, the only way that $\lift{\mu}$ can fail to be embedded is if one of the first two is closed.
\end{proof}

We will now proceed to study the topology of surfaces with the generalised simple lift property.  

\section{The topology of embedded surfaces with the generalised simple lift property}

The geometrical example at the centre of this initial topological study is the double torus minus a disk, that is the connected sum of two tori with a disk removed (see Figure 1).

\begin{center}
\begin{tikzpicture}[scale=0.9]
\node at (0,-2.5) {Figure 1: $\mathbb{T}^2\#\mathbb{T}^2\setminus {D}$};
\draw[xscale=0.8] (-5.5,0) to[out=90,in=180] (-2.75,1.5) to[out=0,in=180] (0,1.25) to[out=0,in=180] (2.75,1.5) to[out=0,in=90] (5.5,0) ;
\draw[rotate=180,xscale=0.8] (-5.5,0) to[out=90,in=180] (-2.75,1.5) to[out=0,in=180] (0,1.25) to[out=0,in=180] (2.75,1.5) to[out=0,in=90] (5.5,0) ;

\draw[xshift=-0.3cm,yshift=-0.1cm,xscale=0.8]  (-3.3,0.1) to[out=65,in=180] (-2.5,0.5) to[out=0,in=115] (-1.7,0.1);
\draw[xshift=-0.3cm,yshift=-0.1cm,xscale=0.8]  (-3.4,0.3) to[out=-75,in=180] (-2.5,-0.2) to[out=0,in=-105] (-1.6,0.3);
\draw[xshift=4.3cm,yshift=-0.1cm,xscale=0.8]  (-3.3,0.1) to[out=65,in=180] (-2.5,0.5) to[out=0,in=115] (-1.7,0.1);
\draw[xshift=4.3cm,yshift=-0.1cm,xscale=0.8]  (-3.4,0.3) to[out=-75,in=180] (-2.5,-0.2) to[out=0,in=-105] (-1.6,0.3);

 \draw[black] (3.8,0) to[out=-90,in=180] (4.2,-0.8) ;
 \draw[black] (3.8,0) to[out=90,in=180] (4.2,0.8);

\end{tikzpicture}

\end{center}

By the classification of compact surfaces, we know that compact orientable surfaces are either the sphere $\mathbb{S}^2$ or the connected sum of $n$ tori, $\mathbb{T}^2\#\cdots\#\mathbb{T}^2$, while non-orientable surfaces are given by the connected sum if $n$ projective planes $\mathbb{R}P^2\#\cdots\#\mathbb{R}P^2$. This classification extends to non-compact surfaces by taking into consideration boundary components.

\begin{oss}
In order to simplify the notation, we will denote by $\mathbb{T}^2_n$ the connected sum of $n$ tori, and by $\mathbb{R}P^2_n$ the connected sum of $n$ projective planes.
\end{oss}

\begin{lemma}\label{no_two_closed} Let $\Sigma$ be an embedded surface with the generalised simple lift property. Then two smooth,
non-separating Jordan curves $\gamma_1$ and $\gamma_2$ intersecting transversally at exactly one point in $\Sigma$ cannot both have a closed $\delta$- lift into any $\Delta=\Delta(\gamma_1,\gamma_2)$, for any $\delta>0$.
\end{lemma}
\begin{proof}
Let $\gamma_1,\,\gamma_2\colon [0,1]\to \Sigma$ be two smooth non-separating Jordan curves intersecting transversally at a single point $p$, that is $p=\gamma_1(0)=\gamma_1(1)=\gamma_2(0)=\gamma_2(1)$. 

Arguing by contradiction, let us assume that both curves admit closed $\delta$-lifts on an embedded disk $\Delta(\gamma_1,\gamma_2)$. In other words, there exists a choice of $\delta>0$ and $U\subset\Sigma$ pre-compact open subset containing both $\gamma_1$ and $\gamma_2$, such that the embedded disk $\Delta=\Delta(\gamma_1,\gamma_2)$ given by Definition \ref{genliftprop} contains $\lift{\gamma_1}$ and $\lift{\gamma_2}$ two closed $\delta$-lifts of $\gamma_1$ and $\gamma_2$ respectively.

Let us then consider the generalised simple lift of ${\gamma_1}\cup{\gamma_2}$ on $\Delta(\gamma_1,\gamma_2)$ pointed at $(p,q)$. We have therefore constructed two simple closed curves $\lift{\gamma_1}$ and $\lift{\gamma_2}$ contained in an embedded disk $\Delta=\Delta(\gamma_1,\gamma_2)$ that intersect transversally in a single point $q\in\Delta(\gamma_1,\gamma_2)$.

This represents
a contradiction to the mod 2 degree theorem applied to the Jordan-Brouwer separation
theorem. This contradiction finishes the proof of the lemma.

\end{proof}

\indent In the following claims, the surface $\Sigma\subset\Omega$ that we are considering is homeomorphic to $\mathbb{T}^2\#\mathbb{T}^2\setminus {D}$ and $\gamma_1:[0,1]\to \Sigma$ denotes the smooth, non-separating Jordan curve in Figure 2 . We will prove that a surface with the generalised simple lift property cannot contain an open subset homeomorphic to a double torus minus a disk by proving that $\gamma_1$ cannot have either a closed or an open lift in an embedded disk $\Delta$ for a specific choice of five non-separating Jordan curves $\gamma_1,\gamma_2,\gamma_3,\gamma_4,\gamma_5$ (see Figure 3).

\begin{center}
\begin{tikzpicture}[scale=0.9]
\node at (0,-2.5) {Figure 2: $\gamma_1$};
\draw[xscale=0.8] (-5.5,0) to[out=90,in=180] (-2.75,1.5) to[out=0,in=180] (0,1.25) to[out=0,in=180] (2.75,1.5) to[out=0,in=90] (5.5,0) ;
\draw[rotate=180,xscale=0.8] (-5.5,0) to[out=90,in=180] (-2.75,1.5) to[out=0,in=180] (0,1.25) to[out=0,in=180] (2.75,1.5) to[out=0,in=90] (5.5,0) ;

\draw[xshift=-0.3cm,yshift=-0.1cm,xscale=0.8]  (-3.3,0.1) to[out=65,in=180] (-2.5,0.5) to[out=0,in=115] (-1.7,0.1);
\draw[xshift=-0.3cm,yshift=-0.1cm,xscale=0.8]  (-3.4,0.3) to[out=-75,in=180] (-2.5,-0.2) to[out=0,in=-105] (-1.6,0.3);
\draw[xshift=4.3cm,yshift=-0.1cm,xscale=0.8]  (-3.3,0.1) to[out=65,in=180] (-2.5,0.5) to[out=0,in=115] (-1.7,0.1);
\draw[xshift=4.3cm,yshift=-0.1cm,xscale=0.8]  (-3.4,0.3) to[out=-75,in=180] (-2.5,-0.2) to[out=0,in=-105] (-1.6,0.3);

\draw[red,thick,dashed] (-4.4,0) to[out=-90,in=180] (-3.8,-0.4) to[out=0,in=-90] (-3,0);
\draw[red,thick] (-4.4,0) to[out=90,in=180] (-3.8,0.4) to[out=0,in=90] (-3,0);
\draw[black] (3.8,0) to[out=-90,in=180] (4.2,-0.8) ;
\draw[black] (3.8,0) to[out=90,in=180] (4.2,0.8);

\end{tikzpicture}

\end{center}

\begin{cl}\label{noclosedlift}
Let $\Sigma\subset\Omega$ be an embedded surface homeomorphic to $\mathbb{T}^2\#\mathbb{T}^2\setminus D$ with the generalised simple lift property, and let us take into consideration the five smooth, non-separating Jordan curves $\gamma_1,\gamma_2,\gamma_3,\gamma_4,\gamma_5\colon[0,1]\to\Sigma$ given in Figure 3. Then $\gamma_1$ does not admit a closed $\delta$-lift on the embedded disk $\Delta=\Delta(\gamma_1,\gamma_2,\gamma_3,\gamma_4,\gamma_5)$ given by Definition \ref{genliftprop}, for any $\delta >0$ and for any $U\subset\Sigma$ pre-compact open set that contains the curves.
\end{cl}

\begin{center}
\begin{tikzpicture}[scale=1.4]
\node at (0,-2.5) {Figure 3: the five loops taken into consideration};
\draw[xscale=0.8] (-5.5,0) to[out=90,in=180] (-2.75,1.5) to[out=0,in=180] (0,1.25) to[out=0,in=180] (2.75,1.5) to[out=0,in=90] (5.5,0) ;
\draw[rotate=180,xscale=0.8] (-5.5,0) to[out=90,in=180] (-2.75,1.5) to[out=0,in=180] (0,1.25) to[out=0,in=180] (2.75,1.5) to[out=0,in=90] (5.5,0) ;

\draw[xshift=-0.3cm,yshift=-0.1cm,xscale=0.8]  (-3.3,0.1) to[out=65,in=180] (-2.5,0.5) to[out=0,in=115] (-1.7,0.1);
\draw[xshift=-0.3cm,yshift=-0.1cm,xscale=0.8]  (-3.4,0.3) to[out=-75,in=180] (-2.5,-0.2) to[out=0,in=-105] (-1.6,0.3);
\draw[xshift=4.3cm,yshift=-0.1cm,xscale=0.8]  (-3.3,0.1) to[out=65,in=180] (-2.5,0.5) to[out=0,in=115] (-1.7,0.1);
\draw[xshift=4.3cm,yshift=-0.1cm,xscale=0.8]  (-3.4,0.3) to[out=-75,in=180] (-2.5,-0.2) to[out=0,in=-105] (-1.6,0.3);

\draw[blue,thick] (-1.7,0) to[out=90,in=180] (0,0.5) to[out=0,in=90] (1.7,0);
\node at (0,.7) [label=center:\textcolor{blue}{$\gamma_3$}]{};
\draw[blue,thick,dashed] (-1.65,0) to[out=-90,in=180] (0,-0.5) to[out=0,in=-90] (1.65,0);
\draw[red,thick,dashed] (-4.4,0) to[out=-90,in=180] (-3.8,-0.4) to[out=0,in=-90] (-3,0);
\node at (-4.7,0) [label=center:\textcolor{red}{$\gamma_1$}]{};
\draw[red,thick] (-4.4,0) to[out=90,in=180] (-3.8,0.4) to[out=0,in=90] (-3,0);
\draw[black] (3.8,0) to[out=-90,in=180] (4.2,-0.8) ;
\draw[black] (3.8,0) to[out=90,in=180] (4.2,0.8);
\draw[yellow,thick] (-3.3,0) to[out=90,in=180] (-2.3,0.8) to[out=0,in=90] (-1.3,0);
\node at (-2.3,1) [label=center:\textcolor{yellow}{$\gamma_2$}]{};

\draw[yellow,thick] (-3.3,0) to[out=-90,in=180] (-2.3,-0.7) to[out=0,in=-90] (-1.3,0);
\draw[green, thick, dashed] (-2.6,-1) to[out=-90,in=180] (-2.3,-1.5) ;
 \draw[green, thick, dashed] (-2.6,-1) to[out=90,in=180] (-2.3,-0.3);
 \draw[green,thick] (-2.3,-1.5) to [out=-90,in=30] (-1.9,-1.3);
 \draw[green,thick] (-2.3,-0.3) to [out=30, in=180](-1.9,-0.8);

\draw[green,thick] (-1.9,-0.8) to [in=180,out=0] (2.3,1)to[out=0,in=90] (3.6,0);
\draw[green,thick] (-1.9,-1.3) to[out=32,in=180] (2.3,-1.1) to[out=0,in=-90] (3.6,0);
\node at (0,-1.05) [label=center:\textcolor{green}{$\gamma_4$}]{};
\draw[purple,thick] (1.3,0) to[out=90,in=180] (2.3,0.8) to[out=0,in=90] (3.3,0);
\node at (2.1,-0.85) [label=center:\textcolor{purple}{$\gamma_5$}]{};

\draw[purple,thick] (1.3,0) to[out=-90,in=180] (2.3,-0.7) to[out=0,in=-90] (3.3,0);
\end{tikzpicture}
\end{center}
\begin{proof}

Given the five smooth Jordan curves $\gamma_1,\gamma_2,\gamma_3,\gamma_4,\gamma_5\colon [0,1]\to\Sigma$ pictured in Figure 3, for an arbitrary $\delta>0$ and for an arbitrary pre-compact open set $U\subset\Sigma$ that contains these five curves $\gamma_i$, we are considering the embedded disk $\Delta=\Delta(\gamma_1,\gamma_2,\gamma_3,\gamma_4,\gamma_5)$ for which the connected component of $\Delta\cap\mathcal{N}_\epsilon(U)$ that contains the $\delta$-lifts $\lift{\gamma_i}$ is a $\delta$-cover of $U$ (see Definition \ref{genliftprop}).

Arguing by contradiction, let us assume that $\gamma_1$ admits a closed $\delta$-lift on such a disk $\Delta=\Delta(\gamma_1,\gamma_2,\gamma_3,\gamma_4,\gamma_5)$.

By Lemma \ref{no_two_closed}, we already know that $\gamma_2$ cannot admit a closed $\delta$-lift on this disk, since $\gamma_1$ and $\gamma_2$ intersect transversally at a single point, and $\gamma_1$ has a closed lift on $\Delta$ by assumption.

Let us now consider the third curve $\gamma_3\colon[0,1]\to\Sigma$.

If $\gamma_3$ admitted an open lift on $\Delta$ then we would have two Jordan curves intersecting transversally at a single point $\lbrace p\rbrace=\gamma_2\cap\gamma_3$, and both admit an open $\delta$-lift on $\Delta=\Delta(\gamma_1,\gamma_2,\gamma_3,\gamma_4,\gamma_5)$. Moreover, we can take as a two-sided pre-compact subset that contains $\gamma_2$ and $\gamma_3$ the original subset $U\subset\Sigma$ used to construct $\Delta$. By the lifting lemma (Proposition \ref{Lifting Lemma}), we know that the curve $\alpha:=\gamma_2\circ\gamma_3\circ\gamma_2^{-1}\circ\gamma_3^{-1}$ admits a closed $\delta$-lift on $\Delta=\Delta(\gamma_1,\gamma_2,\gamma_3,\gamma_4,\gamma_5)$, and that one of the curves $\alpha,\gamma_2\circ\gamma_3$ and $\gamma_3\circ\gamma_2^{-1}$ has an embedded closed lift on $\Delta$.

If either $\gamma_2\circ\gamma_3$ or $\gamma_3\circ\gamma_2^{-1}$ has an embedded closed lift on $\Delta$, then we reach a contradiction by applying the same reasoning used in Lemma \ref{no_two_closed}, since $\gamma_1$ (which we are assuming has a closed lift on $\Delta$) and the given curve intersect transversally in a single point: $\lbrace p\rbrace=\gamma_1\cap\gamma_2\circ\gamma_3$ or $\lbrace p\rbrace=\gamma_1\cap\gamma_3\circ\gamma_2^{-1}$.

If instead $\alpha=\gamma_2\circ\gamma_3\circ\gamma_2^{-1}\circ\gamma_3^{-1}$ admits an embedded closed lift, then one can find three values $t_1,t_2,t_3\in (0,1)$ for which $p=\gamma_1(t_1)=\gamma_2(t_2)=\gamma_2^{-1}(t_3)$. Following the construction of the lifting lemma, we consider the two-sided subset $U\subset\Sigma$ which in particular contains $\gamma_1,\gamma_2$ and $\gamma_3$, and pick a small simply-connected neighbourhood $V$ of $p$ contained in $U$, so that we can construct a family of parallel components of the lifts of $V$ that can be ordered by height: $\Pi_\Sigma^{-1}(V)\cap\Delta=\lbrace \lift{V}(1),\ldots,\lift{V}(n)\rbrace$. We can now consider the generalised simple $\delta$-lift $\lift{\gamma_1}\cup\lift{\alpha}$ of $\gamma_1\cup\alpha$ pointed at $(p,\lift{p})$, where $\lift{p}=\lift{\gamma_1(t_1)}=\lift{\gamma_2(t_2)}\in\mathcal{N}_\epsilon(p)\cap\Delta(\gamma_1,\gamma_2,\gamma_3,\gamma_4,\gamma_5)$. 

The fact that $\lift{p}$ is the only point of intersection results from the following remark. $\lift{\gamma_1}$ is indeed a one-cover of $\gamma_1$, while on the other hand $\lift{\gamma_2^{-1}(t_3)}\in\lift{\alpha}$ belongs to a components of $\Pi_\Sigma^{-1}(V)\cap\Delta$ that is different to that of $\lift{p}=\lift{\gamma_1(t_1)}=\lift{\gamma_2(t_2)}$. In fact, if we denote by $\lift{V}(l_1)$ the component of $\Pi_\Sigma^{-1}(V)\cap\Delta$ that contains $\lift{p}$, we have that the component that contains $\lift{\gamma_2^{-1}(t_3)}$ will have height $l_2$ given by:
\begin{align*}
    l_2=l_1+m[\gamma_3]\neq l_1
\end{align*}
since $\lift{\gamma_3}$ is an open lift on $\Delta(\gamma_1,\gamma_2,\gamma_3,\gamma_4,\gamma_5)$.

Therefore, we constructed two closed curves $\lift{\gamma_1}$ and $\lift{\alpha}$ which intersect transversally on the disk $\Delta=\Delta(\gamma_1,\gamma_2,\gamma_3,\gamma_4,\gamma_5)$ in a single point $\lift{p}$, which represents a contradiction to the mod 2 degree theorem applied to the Jordan Brouwer separation theorem.

Therefore, the $\gamma_3$ must have a closed $\delta$-lift on $\Delta=\Delta(\gamma_1,\gamma_2,\gamma_3,\gamma_4,\gamma_5)$.

Let us now take into consideration the loop $\gamma_4$ which intersects $\gamma_2$ transversally in one single point. Arguing like before, we obtain that $\gamma_4$ will have a closed $\delta$-lift on $\Delta=\Delta(\gamma_1,\gamma_2,\gamma_3,\gamma_4,\gamma_5)$ too.

We have then constructed two smooth non-separating Jordan curves $\gamma_3$ and $\gamma_4$ that intersect transversally in one single point and both have a closed $\delta$-lift on the disk $\Delta$. By Lemma \ref{no_two_closed}, this represents a contradiction to the Jordan Brouwer separation theorem.

This implies that the initial curve $\gamma_1$ cannot have a closed $\delta$-lift on the embedded disk $\Delta=\Delta(\gamma_1,\gamma_2,\gamma_3,\gamma_4,\gamma_5)$. 
\end{proof}
\begin{cl}\label{noopenlift} Let $\Sigma\subset\Omega$ be an embedded surface homeomorphic to $\mathbb{T}^2\#\mathbb{T}^2\setminus D$ with the generalised simple lift property, and let us take into consideration the five smooth, non-separating Jordan curves $\gamma_1,\gamma_2,\gamma_3,\gamma_4,\gamma_5\colon[0,1]\to\Sigma$ given in Figure 3. Then $\gamma_1$ does not admit an open $\delta$-lift on the embedded disk $\Delta=\Delta(\gamma_1,\gamma_2,\gamma_3,\gamma_4,\gamma_5)$ given by Definition \ref{genliftprop}, for any $\delta >0$ and for any $U\subset\Sigma$ pre-compact open set that contains the curves.
\end{cl}
\begin{proof}

Just like in the previous claim, we are working with the same five curves in Figure 3, and for any arbitrary $\delta>0$ and for an arbitrary pre-compact open set $U\subset\Sigma$ that contains these curves $\gamma_i$ we are considering the embedded disk $\Delta=\Delta(\gamma_1,\gamma_2,\gamma_3,\gamma_4,\gamma_5)$ given by Definition \ref{genliftprop}, for which the connected component of $\Delta\cap\mathcal{N}_{\epsilon}(U)$ that contains the $\delta$-lifts $\lift{\gamma_i}$ is a $\delta$-cover of $U$.

Arguing by contradiction like before, we will now assume that $\gamma_1$ admits an open $\delta$-lift on such a disk $\Delta=\Delta(\gamma_1,\gamma_2,\gamma_3,\gamma_4,\gamma_5)$.

Let us consider the curve $\gamma_2\colon[0,1]\to\Sigma$ which intersects $\gamma_1$ transversally in a single point. We will be considering the two cases of $\gamma_2$ admitting either an open or a closed $\delta$-lift on $\Delta=\Delta(\gamma_1,\gamma_2,\gamma_3,\gamma_4,\gamma_5)$, and we will prove that both of them yield a contradiction.

Let us first assume $\gamma_2$ admits a closed $\delta$-lift on $\Delta=\Delta(\gamma_1,\gamma_2,\gamma_3,\gamma_4,\gamma_5)$. By Lemma \ref{no_two_closed}, this implies that both $\gamma_3,\gamma_4\colon[0,1]\to\Sigma$ will have an open $\delta$-lift on $\Delta$, since each one of them intersects transversally $\gamma_2$ in a single point. 

Let us then consider $\gamma_5\colon[0,1]\to\Sigma$, which intersects $\gamma_3$ transversally in a single point. We will see that $\gamma_5$ cannot admit either an open or a closed $\delta$-lift on $\Delta=\Delta(\gamma_1,\gamma_2,\gamma_3,\gamma_4,\gamma_5)$.

If $\gamma_5$ admitted an open lift, the curves $\gamma_3$ and $\gamma_5$ would satisfy the hypotheses of the lifting lemma, and we could then apply the same argument as in the previous claim to the two curves $\gamma_2$ and $\gamma_3\circ\gamma_5\circ\gamma_3^{-1}\circ\gamma_5^{-1}$, both of which have a closed lift on $\Delta$ and intersect transversally at a single point, hence obtaining a contradiction.

If instead $\gamma_5$ admitted a closed lift, then the two curves $\gamma_3$ and $\gamma_4$ would satisfy the hypotheses of the lifting lemma, and we could apply still the same argument as in Claim \ref{noclosedlift} to the two curves $\gamma_5$ and $\gamma_3\circ\gamma_4\circ\gamma_3^{-1}\circ\gamma_4^{-1}$, both of which have a closed lift on $\Delta$ and intersect transversally at a single point, and hence obtain a contradiction. 

These arguments imply that such a curve $\gamma_5$ cannot have either an open or a closed lift on $\Delta$, which is a contradiction, meaning that $\gamma_2$ cannot admit a closed lift.

Let us now study the case where this curve $\gamma_2\colon[0,1]\to\Sigma$ admits an open $\delta$-lift on $\Delta=\Delta(\gamma_1,\gamma_2,\gamma_3,\gamma_4,\gamma_5)$.

This implies that $\gamma_3\colon [0,1]\to\Sigma$ cannot have a closed lift on $\Delta$. In fact, if it did, we would be able to apply the lifting lemma to $\gamma_1$ and $\gamma_2$, and following the same reasoning as in Claim \ref{noclosedlift} to the curves $\gamma_3$ and $\gamma_1\circ\gamma_2\circ\gamma_1^{-1}\circ\gamma_2^{-1}$, we would obtain a contradiction.

Moreover, the fact that $\gamma_3\colon [0,1]\to\Sigma$ admits an open lift on $\Delta$ also implies that $\gamma_5\colon[0,1]\to\Sigma$ has an open lift on $\Delta$. This simply follows from the fact that we would be able to apply the lifting lemma to the two curves $\gamma_2$ and $\gamma_3$. So if $\gamma_5$ had a closed lift on $\Delta$, we would obtain two curves $\gamma_2\circ\gamma_3\circ\gamma_2^{-1}\circ\gamma^{-1}_3$ and $\gamma_5$ intersecting transversally at one single point, both admitting a closed $\delta$-lift on $\Delta$, which is a contradiction as already shown in the previous claim.

We are therefore left to study the scenario where the four curves $\gamma_1,\gamma_2,\gamma_3$ and $\gamma_5$ all have an open lift on $\Delta=\Delta(\gamma_1,\gamma_2,\gamma_3,\gamma_4,\gamma_5)$.

By applying the lifting lemma to the pairs of curves $\lbrace \gamma_1,\,\gamma_2\rbrace$ and $\lbrace \gamma_3,\,\gamma_5\rbrace$ respectively, we obtain two curves with a closed lift on $\Delta=\Delta(\gamma_1,\gamma_2,\gamma_3,\gamma_4,\gamma_5)$, namely $\alpha:=\gamma_1\circ \gamma_2\circ \gamma_1^{-1}\circ \gamma_2^{-1}$ and $\beta:=\gamma_3\circ \gamma_5\circ \gamma_3^{-1}\circ \gamma_5^{-1}$.

As already pointed out in the previous claim, considering the first couple of curves $\lbrace \gamma_1,\gamma_2\rbrace$, one of the curves $\alpha$, $\gamma_1\circ \gamma_2$ and $\gamma_2\circ \gamma_1^{-1}$ has an embedded closed lift, and likewise for the other couple $\lbrace \gamma_3,\gamma_5\rbrace$. In this proof, we will take into consideration only the most complicated case where $\alpha$ and $\beta$ are the loops with the embedded closed lift on $\Delta$. One should argue just like in Claim \ref{noclosedlift} for the other cases.

Let us take into consideration the point of intersection $\lbrace p\rbrace=\alpha\cap\beta$. One should notice that there exist four values $t_1,t_2,t_3,t_4\in (0,1)$ such that
\begin{align*}
p=\gamma_2(t_1)=\gamma_2^{-1}(t_2)=\gamma_3(t_3)=\gamma_3^{-1}(t_4)\,.
\end{align*}

Let us now take as a two-sided pre-compact open set $U\subset \Sigma$ that contains both curves $\alpha$ and $\beta$, the original subset $U\subset\Sigma$ used to construct $\Delta=\Delta(\gamma_1,\gamma_2,\gamma_3,\gamma_4,\gamma_5)$. Following the construction in the lifting lemma, we can pick a small simply-connected neighbourhood $V\subset U$ of $p$, so that we can produce a family of parallel components on $\Delta$ of the lifts of $V$ that can be ordered by height: $\Pi_\Sigma^{-1}(V)\cap\Delta=\lbrace\lift{V}(1),\ldots,\lift{V}(n)\rbrace$.

All the curves $\lift{\gamma_i}$ are the open lifts, so that - still following the notation of the lifting lemma - $m[\gamma_j]\neq 0\;\forall\,j=1,\ldots,4\,$, which means that there are two cases: either $m[\gamma_1]\cdot m[\gamma_5]>0$, or  $m[\gamma_1]\cdot m[\gamma_5]<0$.

In the first case, we will consider the generalised simple lift $\lift{\alpha}\cup\lift{\beta}$ based at $(\gamma_2(t_1)=\gamma_3(t_3),\lift{p})$ on the disk $\Delta$, which means there exists at least a point $\lift{p}\in\mathcal{N}_{\epsilon}(p)\cap\Delta$ such that $\lift{p}\in\lift{\gamma_2}\cap\lift{\gamma_3}$.

We are left to prove that this point $\lift{p}$ is the only point of intersection between $\lift{\alpha}$ and $\lift{\beta}$. By construction, $\lift{\gamma_2(t_1)}$ and $\lift{\gamma_3(t_3)}$ belong to the same component of $\Pi_\Sigma^{-1}(V)\cap\Delta$, namely $\lift{V}(l_1)$. The other two points $\lift{\gamma_2^{-1}(t_2)}$ and $\lift{\gamma^{-1}_3(t_4)}$ will then belong to two other components $\lift{V}(l_2)$ and $\lift{V}(l_3)$ respectively. Moreover, since the number of components between $\lift{\alpha}[k](0)$ and $\lift{\alpha}[k](t)$ does not depend on $k$, we have that the heights of these two components will be:
\begin{align*}
l_2&=l_1-m[\gamma_1]\,,\\
l_3&=l_1+m[\gamma_5]\,.
\end{align*}

Hence $l_3-l_2=m[\gamma_1]+m[\gamma_5]\neq 0\,$, since we assumed $m[\gamma_1]\cdot m[\gamma_5]>0$, which means that $\lift{p}$ is indeed the only point of intersection between the two closed lifts $\lift{\alpha}$ and $\lift{\beta}$, which is a contradiction.

In the second case, where $m[\gamma_1]\cdot m[\gamma_5]<0$, we can repeat the same argument as before, applying it to the generalised simple lift of $\lift{\alpha}\cup\lift{\beta}$ based at $((\gamma_2)(t_1)=\gamma_3^{-1}(t_4),\lift{p})$ instead. By using the same notation as the first case, $\lift{\gamma_2(t_1)}$ and $\lift{\gamma_3^{-1}(t_4)}$ will belong to the same component of $\Pi_\Sigma^{-1}(V)\cap\Delta$, $\lift{V}(l_1)$. The other two points $\lift{\gamma_2^{-1}(t_2)}$ and $\lift{\gamma_3(t_3)}$ will then belong to the components $\lift{V}(l_2)$ and $\lift{V}(l_3)$ respectively. Therefore, the heights of these two components will be given by:
\begin{align*}
l_2&=l_1-m[\gamma_1]\,,\\
l_3&=l_1-m[\gamma_5]\,.
\end{align*}

Therefore $l_3-l_2=m[\gamma_1]-m[\gamma_5]\neq 0\,$, since we assumed $m[\gamma_1]\cdot m[\gamma_5]<0$, which means that $\lift{p}$ is indeed the only point of intersection between the closed lifts $\lift{\alpha}$ and $\lift{\beta}$, so we obtain another contradiction.

Hence we have proved that the curve $\gamma_2\colon [0,1]\to\Sigma$ cannot have either an open or a closed lift on the disk $\Delta$, which is a contradiction, implying that $\gamma_1$ cannot admit an open $\delta$-lift on $\Delta=\Delta(\gamma_1,\gamma_2,\gamma_3,\gamma_4,\gamma_5)$.
\end{proof}
\indent From these claims, we obtain the following result.
\begin{prop}\label{orientableprop}
Given an embedded surface $\Sigma\subset\Omega$ with the generalised simple lift property, $\Sigma$ cannot contain an open subset that is homeomorphic $\mathbb{T}^2\#\mathbb{T}^2\setminus {D}$.
\end{prop}
\begin{proof}
The result follows directly from Claims \ref{noclosedlift} and \ref{noopenlift}.
\end{proof}

By the classification of compact surfaces, we have that orientable surfaces are homeomorphic to $\mathbb{S}^2$ or the connected sum of $n$ tori, $\mathbb{T}^2_n$, while non-orientable compact surfaces are homeomorphic to the connected sum of $n$ projective planes, $\mathbb{R}P^2_n$. Moreover, one should notice that in the non-orientable case we have the following homeomorphisms:
\begin{itemize}
\item $\mathbb{R}P^2_2\cong K$  where $K$ is the Klein bottle, and 
\item $\mathbb{R}P^2_3\cong\mathbb{T}^2\#\mathbb{R}P^2\cong K\#\mathbb{R}P^2$;
\end{itemize}
which means that $\mathbb{R}P^2_{2k}\cong \mathbb{T}^2_{k-1}\#K$ and $\mathbb{R}P^2_{2k+1}\cong\mathbb{T}^2_k\#\mathbb{R}P^2$.

Proposition \ref{orientableprop} then gives the following result.

\begin{cor}\label{FirstResult}
If $\Sigma\subset\Omega$ is an embedded compact surface and has the generalised simple lift property, then it must be topologically $\mathbb{S}^2$, $\mathbb{T}^2$, $\mathbb{R}P^2$,  $\mathbb{R}P^2_2\cong K$, 
 $\mathbb{R}P^2_3\cong\mathbb{T}^2\#\mathbb{R}P^2$ or $\mathbb{R}P^2_4\cong \mathbb{T}^2\#K$.


\end{cor}

\section{Minimal laminations}

Let us now apply the results of the previous section to the case of minimal laminations. Let us first recall some facts about laminations.
\begin{defin} A subset $\mathcal{L}\subset\Omega$ is a \textit{smooth lamination} if for each $p\in\mathcal{L}$, there is a radius $r_p>0$, maps $\phi_p,\psi_p:\mathcal{B}_{r_p}(p)\to \mathbb{B}_1(0)\subset\mathbb{R}^3$ and a closed set $T_p\subset(-1,1)$ with $0\in T_p$ such that:
\begin{itemize}
\item[1)] $\phi_p(p)=\psi(p)=\mathbf{0}$;
\item[2)] $\phi_p$ is a smooth diffeomorphism and $\mathbb{D}_1(0)\subset\phi_p\big(\mathcal{L}\cap\mathcal{B}_{r_p}(p)\big)$;
\item[3)] $\psi_p$ is a Lipschitz diffeomorphism and $\mathbb{B}_1(0)\cap\lbrace x_3=t\rbrace_{t\in T_p}=\psi_p(\mathcal{L}\cap\mathcal{B}_{r_p}(p))$;
\item[4)] $\phi_p^{-1}(\mathbb{D}_1(0))=\psi^{-1}_p(\mathbb{D}_1(0))\,.$
\end{itemize}

We refer to maps $\phi_p$ satisfying properties 1) and 2) as \textit{smoothing maps} of $\mathcal{L}$ and to maps $\psi_p$ satisfying properties 1) and 3) as \textit{straightening maps} of $\mathcal{L}$.
\end{defin}

A smooth lamination $\mathcal{L}\subset\Omega$ is \textit{proper} in $\Omega$ if it is closed, that is $\overline{\mathcal{L}}=\mathcal{L}$. Any embedded smooth surface is a smooth lamination that is proper if and only if the surface is proper.
\begin{defin}
Let $\mathcal{L}\subset\Omega$ be a non-empty smooth lamination. A subset $L\subset\mathcal{L}$ is a \textit{leaf of }$\mathcal{L}$ if $L$ is a connected, embedded surface and for any $p\in L\,,\exists\, r_p>0$ and a smoothing map $\phi_p$ so that $\mathbb{D}_1=\phi_p(L\cap\mathcal{B}_{r_p}(p))\,.$ For each $p\in\mathcal{L}$, we will denote by $L_p$ the unique leaf of $\mathcal{L}$ containing $p$.

A smooth lamination $\mathcal{L}$ is a \textit{minimal lamination} if each one of its leaves is minimal.
\end{defin}

The following is the natural compactness result for sequences of properly embedded minimal surfaces with uniformly bounded second fundamental form (see for instance Appendix B in \cite{cm24} for a proof).
\begin{theorem}
\label{compactness_result_appendixB}
Let $\lbrace\Sigma_i\rbrace_{i\in\mathbb{N}}$ be a sequence of smooth minimal surfaces, properly embedded in $\Omega$. If for each compact subset  $U\subset\Omega$ there is a constant $C(U)<\infty$ so that
\begin{align*}
\sup_{U\cap\Sigma_i}\vert A_{\Sigma_i}\vert\le C(U)\,,
\end{align*}
then, $\forall\,\alpha\in(0,1)$, up to passing to a subsequence, the $\Sigma_i$s converge in $C^{\infty,\alpha}_{loc}(\Omega)$ to $\mathcal{L}$, a smooth proper minimal lamination in $\Omega$.
\end{theorem}
\begin{oss}
While the straightening maps converge in $C^\alpha$, their Lipschitz norms are uniformly bounded on compact subsets of $\Omega$. This follows from the Harnack inequality and is used in the proof of Theorem \ref{compactness_result_appendixB} (see Appendix B of \cite{cm24} and Theorem 1.1 in \cite{sol1}).
\end{oss}

In view of the result in Theorem \ref{compactness_result_appendixB}, one can define the so-called \textit{singular points} of a sequence $\mathcal{S}:=\lbrace\Sigma_i\rbrace_{i\in\mathbb{N}}$ of properly embedded smooth minimal surfaces $\Sigma_i$.
\begin{defin}
Given the sequence $\mathcal{S}=\lbrace\Sigma_i\rbrace$, we define the \textit{regular} points to be the set of points
\begin{align*}
\mathrm{reg}(\mathcal{S}):=\bigg\lbrace p\in\Omega\mid\exists\,\rho>0\text{ such that } \limsup_{i\to\infty}\sup_{B_\rho(p)\cap\Sigma_i}\vert A_{\Sigma_i}\vert <\infty\bigg\rbrace
\end{align*}
and the \textit{singular} points of $\mathcal{S}$ to be the set
\begin{align*}
\mathrm{sing}(\mathcal{S}):=\bigg\lbrace p\in\Omega\mid\forall\,\rho>0\text{ such that } \limsup_{i\to\infty}\sup_{B_\rho(p)\cap\Sigma_i}\vert A_{\Sigma_i}\vert =\infty\bigg\rbrace\,.
\end{align*}
\end{defin}

Clearly, $\mathrm{reg}(\mathcal{S})$ is an open subset of $\Omega$, while $\mathrm{sing}(\mathcal{S})$ is closed in $\Omega$. In general, $\mathrm{sing}(\mathcal{S})\subset\Omega\setminus\mathrm{reg}(\mathcal{S})$ is a strict inclusion, however, by Lemma I.1.4 in \cite{cm24} there exists a subsequence $\mathcal{S}'$ of $\mathcal{S}$ so that $\Omega=\mathrm{reg}(\mathcal{S}')\cup\mathrm{sing}(\mathcal{S}')$. Without loss of generality, we will then consider sequences $\mathcal{S}$ that admit this decomposition.

This work will be centred around limit laminations of minimal disk sequences, so it will be convenient to introduce the following definition (inspired by \cite{WhiteStrucPaper}).
\begin{defin} Let us take a closed set $K\subset\Omega$ in our ambient Riemannian three-manifold $\Omega$. Let us introduce a smooth proper minimal lamination $\mathcal{L}$ in $\Omega\setminus K$ and a sequence $\mathcal{S}=\lbrace\Sigma_i\rbrace_{i\in\mathbb{N}}$ of properly embedded minimal disks in $\Omega$.

We will refer to the quadruple $(\Omega,K,\mathcal{L},\mathcal{S})$ as a \textit{minimal disk sequence} if
\begin{itemize}
\item[i.] $\mathrm{sing}(\mathcal{S})=K$, and
\item[ii.] $\Sigma_i\setminus K$ converge to $\mathcal{L}$ in $C^{\infty,\alpha}_{loc}(\Omega\setminus K)$, for some $\alpha\in (0,1)$.
\end{itemize}
\end{defin}

The case where the $\Sigma_i$ are assumed to be disks has been extensively studied and some structural results have been proved on the possible singular sets $K$ and limit laminations $\mathcal{L}$ of a minimal disk sequence $(\Omega,K,\mathcal{L},\mathcal{S})$. For example, in \cite{cm21,cm22,cm23,cm24} Colding and Minicozzi show that $K$ must be contained in a Lipschitz curve and that for any point $p\in K$ there exists a leaf of $\mathcal{L}$ that extends smoothly across $p$. 

When $\Omega=\mathbb{R}^3$, they further show that either $K=\emptyset$ or $\mathcal{L}$ is a foliation of $\mathbb{R}^3\setminus K$ by parallel planes and that $K$ consists of a connected Lipschitz curve which meets the leaves of $\mathcal{L}$ transversely. Using this result, Meeks and Rosenberg showed in \cite{mr8} that the helicoid is the unique non-flat properly embedded minimal disk in $\mathbb{R}^3$. This uniqueness was then used by Meeks in \cite{me30} to prove that if $\Omega=\mathbb{R}^3$ and $K\neq\emptyset$, then $K$ is a line orthogonal to the leaves of $\mathcal{L}$, which is precisely the limit of a sequence of rescalings of a helicoid.

For an arbitrary Riemannian three-manifold, such a simple description is not possible. In \cite{cm28}, Colding and Minicozzi construct a sequence of properly embedded minimal disks in the unit ball $\mathbb{B}_1(0)\subset\mathbb{R}^3$ which has $K=\lbrace 0\rbrace$ and whose limit lamination consists of three leaves: two non-proper disks that spiral into the third, which is the punctured unit disk in the $x_3$-plane. Inspired by this example, more cases have been constructed where the singular set $K$ consists of any closed subset of a line (\cite{Dean2006,Khan,Kleene,HoffmanWhite1}), as well as examples where $K$ is curved (\cite{mwe1}). Finally, Hoffman and White \cite{hofw2} have also constructed minimal disk sequences in which $K=\emptyset$ and the limit lamination $\mathcal{L}$ has a leaf which is a proper annulus in $\Omega$.



\begin{prop}\label{Leaveshavegenliftprop}
Leaves of a minimal disk sequence in $\Omega$ have the generalised simple lift property.
\end{prop}
\begin{proof}
Given $L$ a leaf of $\mathcal{L}$, if $L$ is a disk, the curves $\gamma_i$ in $L$ are themselves their own simple $\delta$-lifts in any pre-compact open set $U\subset L$ that contains them. Hence the proposition holds trivially, with $q=p$.

In the more general case, when $L$ is not a disk, it is sufficient to prove the existence of a generalised simple lift of a single curve $\gamma$. By Proposition B.1 in Appendix B of \cite{cm24}, we obtain a bound on the Lipschitz norms of the straightening maps, which implies that for each pre-compact open subset $U\subset L$, there is a constant $C=C(U)$ such that $C\lambda\in (0,1)$, and then for each $\Sigma_i\in\mathcal{S}$, $\mathcal{N}_\lambda(U)\cap\Sigma_i$ is a (possibly empty) $C\lambda$-graph over $U$. Given a curve $\gamma:[0,1]\to L$ contained in an open pre-compact subset $U\subset L$, let us denote by $l$ the length of $\gamma$ and  $d$ the diameter of $U$. For any $\delta>0$, choose $\epsilon>0$ such that $C\epsilon <\min\lbrace 1,\delta\rbrace$. Let $\mu=\frac{3}{4}\exp(-2C(l+d))$ and pick $\Sigma_\mu\in\mathcal{S}$ such that $\mathcal{N}_{\mu\epsilon}(p)\cap\Sigma_\mu\neq\emptyset$, where $p=\gamma(0)$. Let $\Gamma$ be a component of $\Sigma_\mu\cap\mathcal{N}_\epsilon(U)$ which contains a point $q\in\mathcal{N}_{\mu\epsilon}(p)\cap\Gamma$. We have chosen $\epsilon>0$ so that $\Sigma_\mu\cap\mathcal{N}_{\epsilon}(U)$ is a $\delta$-graph over $U$. We claim that $\Gamma$ is a $\delta$-cover of $U$ containing a $\delta$-lift of $\gamma$. This follows by showing that any curve in $U$ of length at most $2(l+d)$ starting at $p$ has a lift in $\Gamma$ starting at $q$. By construction, this lift is necessarily a $\delta$-lift. 

Indeed, if $\sigma:[0,T]\to U$ is parametrised by arclength, and $\lift{\sigma}:[0,T']\to\Gamma $ satisfies $\Pi_L(\lift{\sigma}(t))=\sigma(t)$ for some $0<T'<T$, then
\begin{align*}
\bigg\vert\frac{d}{dt} dist^{\Omega}\big(\sigma(t),\lift{\sigma}(t)\big)\bigg\vert\le C\,dist^{\Omega}\big(\sigma(t),\lift{\sigma}(t)\big)
\end{align*}
and so
\begin{align*}
dist^{\Omega}\big(\sigma(t),\lift{\sigma}(t)\big)\le\exp(Ct)\cdot dist^{\Omega}(p,q)<\epsilon\mu\,\exp(Ct)<\epsilon\,,
\end{align*}
where the last inequality follows from the fact that $t\le T\le l+d$. Furthermore, if $t<T$, then the lift $\lift{\sigma}(t)$ may be extended past $t$ provided $dist^\Omega\big(\sigma(t),\lift{\sigma}(t)\big)<\epsilon$, which proves that leaves of a minimal disk sequence have the generalised simple lift property as claimed.
\end{proof}

This result then implies:
\begin{prop}\label{LeafResult}
If $L$ is an embedded compact surface obtained as a leaf  of a minimal disk sequence $(\Omega,K,\mathcal{S},\mathcal{L})$ then it must be topologically $\mathbb{S}^2$, $\mathbb{T}^2$, $\mathbb{R}P^2$,  $\mathbb{R}P^2_2\cong K$, 
 $\mathbb{R}P^2_3\cong\mathbb{T}^2\#\mathbb{R}P^2$ or $\mathbb{R}P^2_4\cong \mathbb{T}^2\#K$.
\end{prop}
\begin{oss}\label{LiftingArgument}
By applying a lifting argument, one can further rule out the sphere $\mathbb{S}^2$ and the projective plane $\mathbb{R}P^2$.
\end{oss}

Combining together Proposition \ref{LeafResult} and the previous remark, we then obtain the following result.

\begin{cor}
 An embedded compact surface $L$ obtained as a leaf of a minimal disk sequence $(\Omega,K,\mathcal{S},\mathcal{L})$ must be topologically $\mathbb{T}^2$, $\mathbb{R}P^2\cong K$, $\mathbb{R}P^2_3\cong\mathbb{T}^2\#\mathbb{R}P^2$ or $\mathbb{R}P^2_4\cong \mathbb{T}^2\#K$.
 \end{cor}


\printbibliography

\end{document}